\documentclass[12pt]{amsart}
\usepackage{epic}
\usepackage[english]{babel}
\usepackage[utf8]{inputenc}     
\usepackage{amssymb,amsmath}
\usepackage{dsfont}
\usepackage[all]{xy}
\advance \textheight by -1 \topmargin
\advance \textheight by -1 \headheight
\advance \textheight by -1 \headsep
\advance \textheight by -2 \footskip
\vsize = \textheight
\hsize = \textwidth
\usepackage{url}
\usepackage{stmaryrd}

\usepackage{tikz}
\usetikzlibrary{automata,positioning,arrows}

\newtheorem{Theorem}{Theorem}[section]
\newtheorem{Lemma}[Theorem]{Lemma}
\newtheorem{Corollary}[Theorem]{Corollary}
\newtheorem{Proposition}[Theorem]{Proposition}
\newtheorem{Definition}[Theorem]{Definition}
\newtheorem{Example}[Theorem]{Example}
\newtheorem{Remark}[Theorem]{Remark}

\usepackage{algorithm,algpseudocode}

\algnewcommand{\IIf}[1]{\State\algorithmicif\ #1\ \algorithmicthen}
\algnewcommand{\EndIIf}{\unskip\ \algorithmicend\ \algorithmicif}
\algnewcommand\algorithmicswitch{\textbf{switch}}
\algnewcommand\algorithmiccase{\textbf{case}}
\algnewcommand\algorithmicassert{\texttt{assert}}
\algnewcommand\Assert[1]{\State \algorithmicassert(#1)}%
\algdef{SE}[SWITCH]{Switch}{EndSwitch}[1]{\algorithmicswitch\ #1\ \algorithmicdo}   {\algorithmicend\ \algorithmicswitch}%
\algdef{SE}[CASE]{Case}{EndCase}[1]{\algorithmiccase\ #1:}{\algorithmicend\     \algorithmiccase}%
\algtext*{EndSwitch}%
\algtext*{EndCase}%

\renewcommand{\c}{\mathfrak{c}}

\newcommand{\dd}{\mathrm{d}_{\pm}}

\newcommand{\Br}{\mathrm{Br}}
\newcommand{\F}{\mathcal{F}}

\newcommand{\MA}{\mathds{A}}
\newcommand{\MI}{\mathds{I}}
\newcommand{\MF}{\mathds{F}}
\newcommand{\MN}{\mathds{N}}
\newcommand{\MZ}{\mathds{Z}}
\newcommand{\MQ}{\mathds{Q}}
\renewcommand{\MR}{\mathds{R}}
\newcommand{\MC}{\mathds{C}}
\newcommand{\MH}{\mathds{H}}

\newcommand{\GL}{\mathrm{GL}}
\newcommand{\SL}{\mathrm{SL}}
\newcommand{\PSL}{\mathrm{PSL}}

\newcommand{\C}{\mathcal{C}} \renewcommand{\c}{\mathfrak{c}}
\newcommand{\br}{\mathrm{Br}}
\newcommand{\End}{\mathrm{End}}

\newcommand{\irr}{\mathrm{Irr}}
\renewcommand{\1}{\mathds{1}}

\newcommand{\ind}{\mathrm{ind}}

\newcommand{\psl}{\mathrm{PSL}}

\DeclareFontFamily{U}{matha}{\hyphenchar\font45}
\DeclareFontShape{U}{matha}{m}{n}{
      <5> <6> <7> <8> <9> <10> gen * matha
      <10.95> matha10 <12> <14.4> <17.28> <20.74> <24.88> matha12
      }{}
\DeclareSymbolFont{matha}{U}{matha}{m}{n}
\DeclareFontFamily{U}{mathx}{\hyphenchar\font45}
\DeclareFontShape{U}{mathx}{m}{n}{
      <5> <6> <7> <8> <9> <10>
      <10.95> <12> <14.4> <17.28> <20.74> <24.88>
      mathx10
      }{}
\DeclareSymbolFont{mathx}{U}{mathx}{m}{n}

\DeclareMathSymbol{\obot}         {2}{matha}{"6B}
\DeclareMathSymbol{\bigobot}       {1}{mathx}{"CB}

\begin{document}
 \title{The orthogonal character table of SL$_2(q)$}
 \author{Oliver Braun}
 \thanks{The first author was supported by the DFG Research Training Group ``Experimental and Constructive Algebra'' at RWTH Aachen University.}
 \email{\url{oliver.braun1@rwth-aachen.de}}
 \author{Gabriele Nebe}
 \email{nebe@math.rwth-aachen.de}
 \address{Lehrstuhl D f\"ur Mathematik, RWTH Aachen University}
 
 
 
 \date{\today}
 
 \begin{abstract}
 The rational 
	invariants of the 
	$\SL_2(q)$-invariant quadratic forms on the real irreducible 
	representations are determined. 
	There is still one open question (see Remark \ref{open}) if
	$q$ is an even square. 
 \end{abstract}
 
 \bibliographystyle{abbrv}

 \maketitle

 \section{Introduction}\label{sec:intro}

Throughout the paper 
let $G$ be a finite group.
The isomorphism classes of  $\MC G$-modules are parametrized by their
characters. Our aim is to extend this connection in order to 
also determine the $G$-invariant quadratic forms from the
character table of $G$. 
The ordinary character table displays the characters $\chi _V$ of
the absolutely irreducible $\MC G$-modules $V$. 
For each $\chi _V$ let
$K$ be the maximal real subfield of the character field of $V$ and 
$W$ the irreducible $KG$-module such that $V$ occurs in $W\otimes _K \MC $. 
Then the space 
\begin{equation*}
	 \F _G(W) := \left\{ F : W\times W \to K \mid \substack{F(v,w) = F(w,v) \text{ and } \\
	 F(gw,gv) = F(w,v) \text{ for all } g\in G, v,w\in W }\right\} 
 \end{equation*}
  of $G$-invariant symmetric bilinear forms on $W$ is at least one-dimensional
  and every non-zero $F\in \F_G(W)$ is non-degenerate.
The character $\chi _V$ also determines the $K$-isometry classes of 
the elements of $\F_G(W)$. 
The orthogonal character table additionally 
contains the invariants (see Section \ref{sec:invar}) that determine the
$K$-isometry classes of $(W,F)$ for 
all non-zero $F\in \F_G(W)$. 

For $G=\SL_2(q)$ the ordinary character table was already known to 
Schur, \cite{schurpsl}.
This paper determines the orthogonal character tables of $\SL_2(q)$ for 
all prime powers $q$. 
For $q=2^n$ with $n$ even and the characters of degree $q+1$
we could not specify which even primes ramify in the Clifford algebra 
	 (see Section \ref{section:sl22}).
 
This work grew out of the first author's PhD thesis \cite{braun} 
written under the supervision of the second author.
 In this thesis, the first author also determines the ordinary 
 orthogonal character tables for all (non-abelian)
 finite quasisimple groups of order up to $200,000$.
 
 \section{Invariants of quadratic spaces}\label{sec:invar}
 
 Let $K$ be a field of characteristic 0, $V$ 
 an $n$-dimensional vector space over $K$ and $F:V\times V\to K$ a 
 non-degenerate symmetric bilinear form.
The two most important invariants attached to such a space 
$(V,F)$ are the discriminant and the Clifford invariant.\\

The {\em discriminant} of $(V,F) $ is
 $$\dd(V,F ):=(-1)^{n(n-1)/2}\det(V,F)$$
 where 
 the determinant $\det(V,F) \in K/(K^\times)^2$ is defined as
 the square class of the determinant of a Gram matrix of $F$ with respect to any basis.
 
 The Clifford algebra $\C(V,F)$ is  the quotient of the tensor algebra 
 by the two-sided ideal $\langle v\otimes v -\frac{1}{2} F(v,v) \cdot 1 ~|~ v \in V \rangle$. 
 A $K$-basis of $\C(V,F)$ is given by the ordered tensors 
 $(\overline{b_{i_1}\otimes \ldots \otimes b_{i_k}} \mid 1\leq i_1 <\ldots < i_k \leq n ) $
 of any basis $(b_1,\ldots , b_n) $ of $V$, in particular $\dim(\C(V,F))=2^n$.
 Put 
 \begin{equation*}
  c(V,F):= \begin{cases} \C(V,F) & \text{ if $n$ is even,}\\
	 \C_0(V,F):= \langle \overline{b_{i_1}\otimes \ldots \otimes b_{i_k}} \mid k \mbox{ even } \rangle  & \text{ if $n$ is odd.}
           \end{cases}
 \end{equation*}
 Then $c(V,F) \cong {\mathcal D}^{r\times r}$ is a central simple $K$-algebra with involution and therefore
 it has order 1 or 2 in the Brauer group.
 The {\em Clifford invariant} of $(V,F)$ is defined as the Brauer
 class of $c(V,F)$: 
 $$
 \c(V,F ) := [c(V,F)] = [{\mathcal D}] \in \br (K) .$$
A more detailed exposition of this material may be found e.g. in \cite{scharlauquadratic}.

Our interest in these two isometry invariants of quadratic spaces is mainly due to the following classical result by  Helmut Hasse.
 
 \begin{Theorem}[\protect{\cite{hasse}}]
  Over a number field $K$ the isometry class of a quadratic space is uniquely determined by its dimension, its determinant, its Clifford invariant and its signature at all real places of $K$.
 \end{Theorem}

 We are mainly interested in the case where $K$ is a number field. 
 Then ${\mathcal D}$ is either $K$ or a quaternion division algebra over $K$. 
 We use two notations for these ${\mathcal D}$, either as a symbol algebra 
 or by giving all the local invariants of ${\mathcal D}$:

 \begin{Definition}\label{quat} 
	 For $a,b\in K$ let  $(a,b):= [ \left(\frac{a,b}{K} \right) ] \in \br(K)$ where 
	 $$\left(\frac{a,b}{K} \right) := \langle 1,i,j,k \mid 
	 i^2=a,j^2=b,ij=-ji=k \rangle .$$
	 By the Theorem of Hasse, Brauer, Noether, Albert 
	 (see \cite[Theorem (32.11)]{reiner})
	 any quaternion algebra ${\mathcal D}$ over $K$ is determined by the set of 
	 places $\wp_1,\ldots ,\wp_s$ (the {\em ramified} places) of $K$, for which the completion of ${\mathcal D}$ 
	 stays a division algebra. 
	 Therefore we also describe
	 ${\mathcal D} = {\mathcal Q}_{\wp_1,\ldots ,\wp_s} $ by 
	 its ramified places, where we assume that the center $K$ is 
	 clear from the context. 
 \end{Definition}

\begin{Example}\label{ex:orthsum}
  Let $(V,F)$ be a bilinear space and $a\in K^\times$.
   Then the scaled space $(V,aF)$ has the following algebraic invariants
	 (see \cite[5.(3.16)]{lam73} for the Clifford invariant): 
    \begin{equation*}
      \dd(V,aF) = \begin{cases}
       \dd(V,F) & \text{ if $\dim(V)$ is even}, \\
       a\dd(V,F) & \text{ if $\dim(V)$ is odd}.
                 \end{cases}
 \end{equation*}
  and
  \begin{equation*}
  \c(V,aF)=\begin{cases}
    \c(V,F) (a, \dd(V,F)) & \text{ if $\dim(V)$ is even}, \\
   \c(V,F) & \text{ if $\dim(V)$ is odd}.
     \end{cases}
 \end{equation*}
	If  $$(V,F) = (V_1,F_1) \perp (V_2,F_2) $$ 
	is the orthogonal direct sum of two subspaces 
	the determinant is just the 
	product $\det(V,F) = \det(V_1,F_1) \cdot \det(V_2,F_2)$.\\
	The behavior of the Clifford invariant is more complicated, cf. \cite{lam73}: $\c(V,F) =$
 \begin{equation*}
 \begin{cases}
	 \c(V_1,F_1)\c(V_2,F_2)(\dd(V_1,F_1) , \dd(V_2,F_2)), &  \dim(V_1) \equiv \dim(V_2) \pmod{2},\\
                               \c(V_1,F_1)\c(V_2,F_2)(-\dd(V_1,F_1) , \dd(V_2,F_2)), &  \dim(V)\equiv \dim(V_1) \equiv 1 \pmod{2}.
                              \end{cases}
\end{equation*}
\end{Example}

\begin{Example}\label{ex:An}
Let $\MI _n$ be the $n$-dimensional $\MQ$-vector space that has an 
orthonormal basis $(e_1,\ldots , e_n)$.
Then $\dd (\MI_n) = (-1) ^{n(n-1)/2} (\MQ ^{\times })^2$ and 
	   \begin{equation*}
		     \c(\MI_n)=\begin{cases}
			                           (1,1) & n\equiv 0,1,2,7 \pmod{8}, \\
			                           (-1,-1) & n\equiv 3,4,5,6 \pmod{8}.
			                          \end{cases}
						   \end{equation*}
The  space 
$\MA _{n-1} :=  \langle \sum _{i=1}^n e_i \rangle ^{\perp } \leq \MI_n$ 
is the orthogonal complement of  a space of discriminant $n$ in $\MI_n$. 
This allows to compute 
the discriminant and Clifford invariant of 
 $\MA_{n-1} $
 using the formulas from the previous example: 
 $\dd( \MA_{n-1})=(-1)^{(n-1)(n-2)/2} n (\MQ ^{\times })^2$ and $\c(\MA_{n-1})$ depends 
 on the value of $n$ modulo $8$:
 \begin{center}
	 \begin{tabular}{|c||c|c|c|c|}
		 \hline 
		 $n\pmod{8} $ & $0,1 $ & $2,3 $ & $4,5$ & $6,7 $ \\ 
		 \hline 
		 $\c(\MA_{n-1}) $ & $1$ & $(-1,n) $ & $(-1,-1)$ & $(-1,-n)$ \\
		 \hline
	 \end{tabular}
 \end{center}
\end{Example}
 
\section{Methods} 

\subsection{Orthogonal character tables}  \label{orthchartable} 

	Let $\chi $ be a complex irreducible character of the 
	finite group $G$ and let $K= \MQ(\chi )^+$ be the maximal real subfield
	of the character field $\MQ(\chi )$. 
Let $V$ be 
the  irreducible $\MC G$-module affording the character $\chi $
and
let $W$ be the irreducible $KG$-module such that $V$ is a constituent of 
$W_{\MC }:= \MC \otimes_{K} W$.
Put
\begin{equation*}
	 \F _G(W) := \left\{ F : W\times W \to K \mid \substack{F(v,w) = F(w,v) \text{ and } \\
	 F(gw,gv) = F(w,v) \text{ for all } g\in G, v,w\in W }\right\}
 \end{equation*}
 the space of $G$-invariant symmetric bilinear forms on $W$.
 As $W$ is irreducible, all non-zero elements of $\F_G(W)$ are non-degenerate
 and an easy averaging argument shows that $\F_G(W)$ always contains a
 totally positive definite form $F_0$. 
 We call $W$ {\em uniform} if $\F_G(W) = \langle F_0 \rangle _{K }$ is 
one-dimensional  over $K$.

\begin{Remark}\label{three} 
	 There are three different situations to be considered:
	 \begin{itemize}
		 \item[(a)] $K=\MQ(\chi ) $ and $V=W_{\MC} $: 
			 Then $W$ is an absolutely irreducible $KG$-module and 
			 hence uniform. 
		 \item[(b)] $K=\MQ (\chi )$ and $W_{\MC} \cong V\oplus V $: 
			 Then the Schur index of $\chi $ over $K$ is $2$,
			 $\chi (1)$ is even, 
			 and \cite{turull} tells us that $\dd (F) \in (K^{\times })^2 $
			 for all non-zero $F\in \F_G(W) $. 
If the real Schur index of $\chi $ is one, then  $\dim (\F_G(W)) = 3$.
			 If the real Schur index of $\chi $ is $2$, 
			 then $W$ is uniform and \cite[Theorem B]{turull} 
			 also gives the Clifford invariant of $(W,F)$: 
		 $$\c (W,F) = 
	 \begin{cases} 1 & \mbox{ if } \dim _K(W) \equiv 0 \pmod{8} \\ {[}\End _{KG}(W) ] & \mbox{ if } \dim _K(W) \equiv 4 \pmod{8} . \end{cases} $$
		 \item[(c)] $[\MQ(\chi ) : K] = 2$. Then 
			 $\chi _W = m(\chi + \overline{\chi })$ for some $m\in \MN$ and
			 $W$ is uniform if and only if $m=1$. 
			Choose $\delta \in K$ such that $\MQ(\chi ) = K(\sqrt{\delta })$,
			then $\dd (F) = \delta ^{m \chi(1)}  (K^{\times })^2 $ 
			for all $0\neq F\in \F_G(W)$
			(see \cite[Theorem 10.1.4]{scharlauquadratic}, \cite[Theorem 4.3.9]{braun}). 
	\end{itemize}
 \end{Remark}

\begin{Definition}
	Let $\chi $, 
 $K:=\MQ (\chi )^{+}$,  $W$ be as above. Put $n:=\dim_K(W)$ and 
 choose $0\neq F\in \F_G(W) $.
	If $n$ is even then we define 
	$$\dd (\chi ) := \dd (W,F) .$$ 
	If $n$ is odd, or $n$ is even, $W$ is uniform,  and $\dd (\chi ) = 1$, 
	then we put 
	$$\c (\chi ) := \c (W,F) .$$
	The {\em orthogonal character table} of $G$ 
	is the complex character table of $G$ with this additional information added.
\end{Definition}

By Example \ref{ex:orthsum} and 
Remark \ref{three}  the values $\dd (\chi )$  and $\c (\chi )$ are well defined, i.e. 
independent of the choice of the non-zero $F\in \F_G(W)$.


\subsection{Clifford orders} 

Let us now assume that $K$ is a local or global field of characteristic $0$,
i.e. $K$ is a finite extension of either $\MQ _p$ or $\MQ $, and let 
$R$ denote the ring of integers in $K$.  
Let 
$V$ be a finite dimensional vector space over $K$ and $F:V\times V \to K$ a symmetric bilinear
form with associated quadratic form 
$$Q_F:V\to K,  v\mapsto Q_F(v) = \frac{1}{2} F(v,v) .$$

\begin{Definition}\label{lattices}
	A {\em lattice} $L$ in $V$ is a finitely generated $R$-submodule of $V$ that 
	contains a $K$-basis of $V$. 
	The lattice $L$ is called {\em integral}, if $F(L,L) \subseteq R$ and 
	{\em even}, if $Q_F(L) \subseteq R$. 
	The {\em dual lattice} of $L$ is $L^{\#} := \{ v\in V \mid F(v,L) \subseteq R \} $ 
	and $L$ is called {\em unimodular}, if $L=L^{\#} $. 
\end{Definition}

        Even unimodular lattices are called regular quadratic $R$-modules in \cite{Kneser}.
	If $2\not\in R^{\times }$, then there are no regular $R$-modules  $L$ of odd dimension.
	Kneser calls an even lattice $L$ of odd dimension such that $L^{\#} /L \cong R/2R $ 
	a semi-regular quadratic $R$-module.

	\begin{Theorem} [\protect{\cite[Satz 15.8]{Kneser}}]
		Assume that $R$ is a complete discrete valuation ring (with finite 
		residue class field) and let $L$ be a regular or semi-regular
		quadratic $R$-module  in $(V,Q_F)$. 
		If $\dim (V) \geq 3$ then $L\cong \MH(R) \perp M$ 
		for some regular or semi-regular quadratic $R$-module $M$. 
		Here $\MH(R)$ is the hyperbolic plane, the regular free $R$-lattice with 
		basis $(e,f)$ such that $Q_F(e)=Q_F(f) = 0$ and $F(e,f) = 1$.
	\end{Theorem}

	As both invariants, the Clifford invariant and the discriminant of the
	hyperbolic plane $\MH (K) = K \MH (R)$ are trivial, we obtain the following
	corollary.

	\begin{Corollary}\label{semireg}
		Under the assumption of the theorem let $\dim (V)$ be odd and 
		$L$ be a semi-regular lattice in $V$. 
		Then $\c (V,F) = 1 $.
	\end{Corollary}

	\begin{proof}
		We proceed by induction on the dimension of $V$.
		If $\dim(V) = 1$ then $c(V,F) =K $ and so $\c(V,F) = 1$. 
		So assume that $\dim (V) \geq 3$. Then 
		$L \cong \MH(R) \perp M$ and hence $V\cong \MH(K) \perp KM $ 
		for some semi-regular lattice $M$ in $KM$. 
		By induction we have $\c(KM,F_{|KM}) = 1$. 
		So 
		$$\c(V,F) = \c(KM, F_{|KM} ) \c (\MH(K))  (-\dd (KM,F_{|KM}) , \dd (\MH(K) ) )  = 1.
		\hfill{ \qedhere } 
		$$
	\end{proof}

	\begin{Remark}\label{detCO}
Let $L$ be an even lattice in $V$. 
Then the {\em Clifford order} $\C (L,F) $ of $L$ is defined to be the 
$R$-subalgebra of $\C(V,F)$ generated by $L$. 
As $Q_F(L) \subseteq R$, the Clifford order is an $R$-lattice in $\C(V,F)$, in particular
finitely generated over $R$.
If $L$ has an orthogonal basis $(b_1,\ldots , b_n)$, then 
the ordered tensors
 $(\overline{b_{i_1}\otimes \ldots \otimes b_{i_k}} \mid 1\leq i_1 <\ldots < i_k \leq n ) $
 form an $R$-basis of $\C(L,F)$.
 In this case it is easy to compute the determinant of $\C(L,F)$ and of $\C_0(L,F)$ with respect to 
 the reduced trace bilinear form (see \cite[Theorem 7.2.2]{braun}): 
 Up to some power of $2$ they are both powers of $Q_F(b_1)\cdots Q_F(b_n) $.
\end{Remark}

\begin{Corollary} \label{odddetur}
Assume that  $K$ is a number field,
$2\neq p \in \MZ$ is some  odd prime and 
 $\wp $ is  a prime ideal  of $K$ containing $p$.
 Denote the completion of $K$
 at $\wp $ by $K_{\wp}$ and its valuation ring by $R_{\wp }$. 
Assume that there is a lattice $L$ in $V$ such that
$L_{\wp } = R_{\wp }\otimes L$ is an even  unimodular $R_{\wp }$-lattice.
Then
$$[c (V,F) \otimes K_{\wp }] = 1 \in \Br (K_{\wp }).$$
\end{Corollary}

\begin{proof}
	Since $2\in R_{\wp }^{\times }$ the lattice $L_{\wp }$ has an orthogonal 
basis and Remark \ref{detCO} shows that  the determinant of 
the Clifford order $\C (L_{\wp },F)$ and also of $\C_0(L_{\wp },F)$ is a unit
in $R_{\wp }$. In particular the determinant of a maximal order in 
$c(V,F) \otimes K_{\wp }$ is a unit in $R_{\wp }$,  which shows that this central 
simple $K_{\wp }$-algebra is a matrix ring over $K_{\wp }$ (see for instance \cite[Theorem (20.3)]{reiner}).
\end{proof}

A bit more generally we may also compute the Clifford invariant of a 
bilinear space that contains a lattice of prime determinant: 

\begin{Corollary}\label{primedet}
	Keep the assumptions of Corollary \ref{odddetur} and let 
	$(W_{\wp },E_{\wp })$ be a 1-dimensional bilinear $K_{\wp }$ vector space 
	such that the $\wp $-adic valuation of the discriminant of $E_{\wp }$ is odd. 
	Then 
	$$\c((V\otimes K_{\wp } , F) \perp (W_{\wp },E_{\wp }))  = 1 \in \Br (K_{\wp }) 
	\mbox{ if and only if } \dd (V\otimes K_{\wp } , F) \in (K_{\wp }^{\times })^2 .$$
\end{Corollary}

\begin{proof}
	Clearly the Clifford invariant of the 1-dimensional space is trivial 
	and also $\c(V\otimes K_{\wp } , F) $ is trivial by Corollary \ref{odddetur}.
	So the formula in Example \ref{ex:orthsum} gives us the Clifford invariant of the
	orthogonal sum as 
	$$ \c((V\otimes K_{\wp } , F) \perp (W_{\wp },E_{\wp }) ) = 
	(\dd(V\otimes K_{\wp } , F) , u \pi ) $$ 
	where $u$ is a unit and $\pi $ is a prime element in the valuation ring $R_{\wp }$. 
	As $d:=\dd(V\otimes K_{\wp } , F)\in R_{\wp }^{\times }$, 
	this quaternion symbol is trivial if and only if $ d$ is a square. 
\end{proof}

\subsection{A Clifford theory of orthogonal representations}

Let $N\trianglelefteq G$ be a normal subgroup. 
Clifford theory explains the interplay between irreducible representations of $N$ and $G$
(see for instance \cite[Section 11.1]{curtisreiner1}). 
We want to describe the behavior of invariant forms under this correspondence.

Let $K$ be a totally real number field and $V$ an irreducible $KG$-module with 
a non-degenerate  invariant form $F$.  
We will then call $(V,F)$ an {\em orthogonal representation} of $G$.
Let $U$ be an irreducible  $KN$-module occurring as a direct summand of $V|_N$
with multiplicity $e$. 
Let $I$ be the inertia group of $U$, of index $t:=[G:I]$ in $G$, 
and let $G=\bigsqcup_{i=1}^t g_i I$ be a decomposition of $G$ into left $I$-cosets. 
We then have the following decomposition  of $V|_{N}$ 
into pairwise non-isomorphic irreducible $KN$-modules 
${^{g_i}}U $  ($i=1,\ldots ,t$):

\begin{equation}\label{clifforddecomp}
 V|_N \cong \bigoplus_{i=1}^t \left( {^{g_i}}U \right)^e ,
\end{equation}

In this situation we obtain the following theorem

\begin{Lemma}\label{orthogonalclifford}
	 The 
	 decomposition \eqref{clifforddecomp} is orthogonal 
	 $$(V|_N,F)  =
		 \left( {^{g_1}}U ^e , F_1 \right)  \perp 
		 \left( {^{g_2}}U ^e , F_2 \right)  \perp \ldots \perp 
		 \left( {^{g_t}}U ^e , F_t \right)  $$
		 and 
the forms $F_i$ are non-degenerate and pairwise $K$-isometric.
\end{Lemma}

\begin{proof}
   Clearly, the restriction of $F$ to the 
   direct summand ${^{g_i}}U^e$ is $N$-invariant. 
For $i\neq j$ we have 
$$ {^{g_i}}U  \cong  {^{g_i}}U ^*  \not\cong  {^{g_j}}U $$
so the summands
 $\left( {^{g_i}}U \right)^e$ are orthogonal to each other and the $F_i$ are 
 non-degenerate. 
 The elements $g_j^{-1} g_i \in G \leq O(V,F)$ induce isometries 
 between $F_i$ and $F_j$.
\end{proof}

\begin{Example}\label{borel}
 Consider an odd prime $p$, a natural number $n$ and abbreviate $q:=p^n$. 
 Let $C_{(q-1)/2} \cong H\leq \GL_n(\MF_p)$ be a 
 subgroup acting with regular orbits on $\MF_p^n \setminus \{ 0 \} $ in its 
 natural action.
 Then the group $G:=C_p^n \rtimes H$, which is isomorphic to the normalizer 
 of a Sylow $p$-subgroup in $\psl_2(q)$
 has $(q-1)/2$  linear characters and 
 two non-linear characters $\psi _1, \psi_2 $ of degree $(q-1)/2$ 
 with Schur index 1 and character field 
 \begin{equation*}
  \MQ(\psi_1) = \MQ(\psi_2) =
  \begin{cases}
   \MQ(\sqrt{q}) & \text{ if } q\equiv 1 \pmod{4}, \\
   \MQ(\sqrt{-q}) & \text{ if } q\equiv 3 \pmod{4}.
  \end{cases}
 \end{equation*}
 Let $H_1$ be the unique subgroup of $H$ of order $\frac{p-1}{2}$ and put 
 $N:=C_p^n \rtimes H_1$. Then $N\trianglelefteq G$ and we will apply
 Theorem \ref{orthogonalclifford} to this normal subgroup in order to compute the
 discriminant $\dd(\psi_i)$ in the case $q\equiv 1 \pmod{4}$.\\
 Let $\psi\in\{\psi_1 , \psi_2\}$, $K=\MQ(\psi) = \MQ(\sqrt{q}) $ 
 and $(V,F)$ an orthogonal $KG$-module 
 whose character is $\psi$. \\
 There is a character $\1 \neq \chi \in \irr(C_p^n)$ such that 
 $\psi=\ind_{C_p^n}^G ( \chi) = \ind_N^G ( \ind_{C_p^n}^N (\chi))$. 
 Ordinary Clifford theory shows that $\kappa:=\ind_{C_p^n}^N (\chi)$ is irreducible and an easy application of Frobenius reciprocity 
 reveals $(\psi|_N , \kappa)_N =1$. \\
 Thus we obtain an orthogonal decomposition
 \begin{equation*}
  (V|_N,F) \cong (V_1,F_1) \perp ... \perp (V_t , F_t) 
 \end{equation*}
 where $F_1 \cong ... \cong F_t$ by Lemma \ref{orthogonalclifford}. 
 We have $t=\frac{1}{2} \frac{q -1}{p-1}$ if $K=\MQ$ and 
 $t=\frac{q -1}{p-1}$ if $K=\MQ(\sqrt{p})$.\\
 Notice that $\kappa$ is a faithful character of a group isomorphic to $C_p \rtimes C_{\frac{p-1}{2}}$. 
 As the trace forms of cyclotomic fields 
 are well understood (cf.\ \cite[Section 3.3.2]{nebehabil}), 
 we can find the determinants of the $(V_i,F_i)$ 
 as
 $$\det (V_i,F_i) = \det(V_1,F_1) = \begin{cases} p(\MQ^\times)^2 & \text{ if } n \text{ is even } \\
 u \sqrt{p} ( \MQ(\sqrt{p})^\times)^2  & \text{ if } n \text{ is odd} \end{cases} $$
 for some unit $u$ of the ring of integers of $\MQ(\sqrt{p})$.
 In conclusion, we obtain
 \begin{equation*}
  \det(\psi) = 
  \begin{cases}
   1(\MQ^\times)^2 & \text{ if $n\equiv 0 \pmod{4}$ or $p\equiv 3 \pmod{4}$,} \\
   p(\MQ^\times)^2 & \text{ if $n\equiv 2 \pmod{4}$ and $p\equiv 1 \pmod{4}$,} \\
   u \sqrt{p} ( \MQ(\sqrt{p})^\times)^2 & \text{ if $n \equiv 1 \pmod{2}$,}
  \end{cases}
 \end{equation*}
 In the case $q\equiv 3 \pmod{4}$ the character $\psi$ has non-real values and we find $\dd(\psi) = -p (\MQ^\times)^2$.
\end{Example}

\section{The orthogonal character table of  $\SL_2(q)$ for odd  $q$}

Let $p$ be an odd prime, $n$ a natural number,  put $q:=p^n$ 
and let $G:=\SL_2(q)  $ be the group of all $2\times 2$ matrices of determinant 1
over the field with $q$ elements. 
A reference for the ordinary (and modular) representation theory of this group is, for example \cite{bonnafe}.
We use the ordinary character table and the notation of the absolutely irreducible characters from \cite{dornhoff}:

\begin{Theorem}[\protect{\cite[Theorem 38.1]{dornhoff}}]\label{slct}
 Let $\langle \nu \rangle = \MF_q^\times$. Consider
   \begin{equation*}
	      1=\begin{pmatrix}1&0 \\ 0&1 \end{pmatrix}, ~ z=\begin{pmatrix}-1&0 \\ 0&-1 \end{pmatrix} , ~ c=\begin{pmatrix}1&0 \\ 1&1 \end{pmatrix} , ~ d=\begin{pmatrix}1 & 0 \\ \nu & 1\end{pmatrix}, ~ a =\begin{pmatrix}\nu & 0 \\ 0 & \nu^{-1} \end{pmatrix}
	     \end{equation*}
and let $b\in \SL_2(q)$ be an element  of order $q+1$.\\
	         For $x\in \SL_2(q)$, let $(x)$ denote the conjugacy class containing $x$. $\SL_2(q)$ has the following $q+4$ conjugacy classes of elements, listed together with the size of the classes.
		   \begin{center}
	     \def\arraystretch{1.3}
	       \begin{tabular}{c||c|c|c|c|c|c|c|c}
         $x$ & $1$ & $z$ & $c$ & $d$ & $zc$ & $zd$ & $a^\ell$ & $b^m$ \\
         \hline $|(x)|$ & $1$ & $1$ & $\frac{1}{2}(q^2-1)$ & $\frac{1}{2}(q^2-1)$ & $\frac{1}{2}(q^2-1)$ & $\frac{1}{2}(q^2-1)$ & $q(q+1)$ & $q(q-1)$
	         \end{tabular}
		   \end{center}
 where $1\leq \ell \leq \frac{q-3}{2}$, $1\leq m \leq \frac{q-1}{2}$.
 \\
 Put $$\varepsilon:=(-1)^{(q-1)/2},\ \zeta _r:= \exp (2\pi i/r)  \text{ and } \vartheta _r^{(s)} := \zeta _r^s+\zeta_r^{-s} \mbox{ for } r,s\in \MN .$$ 
 \newpage
 Then the character table of $\SL_2(q)$ reads
as
	   \begin{center}
		      \def\arraystretch{1.3}
		         \begin{tabular}{c|cccccc}
				     & $1$ & $z$ & $c$ & $d$ & $a^\ell$ & $b^m$ \\
				     \hline $\1$ & $1$ & $1$ & $1$ & $1$ & $1$ & $1$ \\
				 $\psi$ & $q$ & $q$ & $0$ & $0$ & $1$ & $-1$ \\
				 $\chi_i$ & $q+1$ & $(-1)^i (q+1)$ & $1$ & $1$ & $\vartheta_{q-1}^{(i\ell)} $ & $0$ \\
				 $\theta_j$ & $q-1$ & $(-1)^j (q-1)$ & $-1$ & $-1$ & $0$ & $-\vartheta_{q+1}^{(jm)} $ \\
				 $\xi_1$ & $\frac{1}{2}(q+1)$ & $\frac{1}{2}\varepsilon (q+1)$ & $\frac{1}{2} \left( 1+\sqrt{\varepsilon q} \right)$ & $\frac{1}{2} \left( 1-\sqrt{\varepsilon q} \right)$ & $(-1)^\ell$ & $0$ \\
				 $\xi_2$ & $\frac{1}{2}(q+1)$ & $\frac{1}{2}\varepsilon (q+1)$ & $\frac{1}{2} \left( 1-\sqrt{\varepsilon q} \right)$ & $\frac{1}{2} \left( 1+\sqrt{\varepsilon q} \right)$ & $(-1)^\ell$ & $0$ \\
				 $\eta_1$ & $\frac{1}{2}(q-1)$ & $-\frac{1}{2}\varepsilon (q-1)$ & $\frac{1}{2} \left( -1+\sqrt{\varepsilon q} \right)$ & $\frac{1}{2} \left( -1-\sqrt{\varepsilon q} \right)$ & $0$ & $(-1)^{m+1}$ \\
				 $\eta_2$ & $\frac{1}{2}(q-1)$ & $-\frac{1}{2}\varepsilon (q-1)$ & $\frac{1}{2} \left( -1-\sqrt{\varepsilon q} \right)$ & $\frac{1}{2} \left( -1+\sqrt{\varepsilon q} \right)$ & $0$ & $(-1)^{m+1}$
				    \end{tabular}
				      \end{center}
					  where $1\leq i \leq \frac{q-3}{2}$, $1\leq j \leq \frac{q-1}{2}$, $1\leq \ell \leq \frac{q-3}{2}$, $1\leq m \leq \frac{q-1}{2}$.\\
					    The columns for the classes $(zc)$ and $(zd)$ are omitted because for any irreducible character $\chi$ the relation
						         $\chi(zc)= \frac{\chi(z)}{\chi(1)}\chi(c) $
					     holds.
	      \end{Theorem}

\begin{Theorem}\label{thm:mainresult}
	The following table 
	gives the orthogonal character table of $\SL_2(q)$. 
   \begin{center}
   \def\arraystretch{1.3}
    \begin{tabular}{|c||c|c|c|c|c|}
     \hline 
     $\chi $ & $K$ & $\dim_K(W) $ & $\c(\chi )$ & $\dd (\chi ) $ & $q$ \\ 
     \hline 
     \hline 
     $\1$ & $\MQ $ & $1$ & $1 $ & $-$ & all \\ 
     \hline
     $\psi $ & $\MQ $ & $q$ & $\c(\MA_q) $ & $-$ & all \\ 
     \hline
     $\substack{\chi _i \\ i \text{ even} }  $ & $\MQ (\vartheta _{q-1}^{(i)}) $ & $q+1$ & $- $ & $\varepsilon (\vartheta_{q-1}^{(2i)}-2)q $ & all \\ 
\hline       
$\substack{\chi _i \\ i \text{ odd } }  $ & $\MQ (\vartheta _{q-1}^{(i)}) $ & $2(q+1)$ & $[\End_{KG}(W)]  $ & $1$ & $1 \pmod{4} $ \\ 
&  &  & $1$ & $1$ & $3 \pmod{4} $ \\ 
\hline 
$\substack{\theta _j \\ j \text{ even} }  $ & $\MQ (\vartheta _{q+1}^{(j)}) $ & $q-1$ & $1 \text{ if } q=\square  $ & $\varepsilon q $ & all \\ 
\hline 
$\substack{\theta _j \\ j \text{ odd} }  $ & $\MQ (\vartheta _{q+1}^{(j)}) $ & $2(q-1)$ & $1$ & $1$ & $1\pmod{4} $  \\ 
	&  &  & $[\End_{KG}(W)]  $ & $1$ & $3 \pmod{4} $ \\ 
\hline 
$\xi _1,\xi _2 $ & $\MQ (\sqrt{q}) $ & $\frac{q+1}{2} $ & $1$ & $-$ & $q\equiv 1,-3 \pmod{16} $ \\ 
 &  &  & $[{\mathcal{Q}}_{p,\infty }\otimes K] $ & $- $ & $q\equiv 5,9 \pmod{16} $ \\
\hline 
$\xi _1 =\overline{\xi _2} $ & $\MQ $ & $q+1$ & $1$ & $1$ & $ 3 \pmod{8} $  \\ 
				&  &  & $(-1,-1)$ & $1$ & $ 7 \pmod{8} $  \\ 
\hline
$\eta _1, \eta _2 $ & $\MQ(\sqrt{q} )  $ & $q-1$ & $[{\mathcal{Q}}_{p,\infty } \otimes K ] $ & $1$ & $1\pmod{4} $ \\
\hline
$\eta _1 =\overline{ \eta _2} $ & $\MQ $ & $q-1$ & $-$ & $-q$ & $3\pmod{4} $ \\
\hline
\end{tabular}
\end{center}
	We use the abbreviations introduced in Theorem \ref{slct}.
As before $K$ is the maximal real subfield of the character field and 
$W$ the irreducible $KG$-module, whose character contains $\chi $.
\end{Theorem}
 
\section{The proof of Theorem \ref{thm:mainresult}} 
 
 \subsection{The faithful characters of $G$}

	The faithful irreducible characters of $\SL_2(q)$ either have real Schur index $2$ or they take values in an imaginary quadratic number field. 
	Janusz \cite[Theorem]{janusz} contains an explicit description of the 
	endomorphism rings $\End_{KG}(W)$.
 In particular their discriminants and Clifford invariants can be read off from Remark \ref{three} 
 (b) and (c). 

 \subsection{The non-faithful characters $\eta _i$} 

 If $q\equiv 3 \pmod{4}$ then the characters $\eta _1$ and $\eta _2$ 
 of degree $(q-1)/2$ have character field $\MQ (\sqrt{\epsilon q}) = \MQ (\sqrt{-p}) $
 and Schur index 1. So Remark \ref{three} (c) yields their discriminant. 

 \subsection{The Steinberg character} 

 The character $\psi $ is a non-faithful character of degree $q$. 
 As $\1 + \psi $ is the character of a 2-transitive permutation representation
 of $G$, the invariants of $\psi $ are those of 
 $\MA _q $ as given in Example \ref{ex:An}.

 \subsection{The characters $\theta _j$, $j$ even} \label{sectheta} 

 For even $j$, the character $\theta _j$ is a non-faithful character of 
 even degree $q-1$ with totally real character field $K$ and Schur index 1. 
 Let $(W,F)$ be the orthogonal $KG$-module affording the character $\theta _j$. 
 Then the restriction of $W$ to the Borel subgroup 
 $B \cong (C_p)^n \rtimes C_{(q-1)/2}  $ of $\psl_2(q) $ 
 has character $\psi_1 + \psi_2 $ from Example \ref{borel}. 
 As $\dd (\psi _1) $ and $ \dd (\psi _2) $ are Galois conjugate, 
 the formula for $\dd (\psi _1 )$ in Example \ref{borel} 
 yields 
 $$\dd (\theta _j) = \begin{cases} 
	 1 (K^\times )^2 & n \mbox{ even  }  \\
	 \varepsilon  p (K^\times )^2 & n \mbox{ odd.} 
 \end{cases} 
 $$
 If $n$ is even then we can also deduce the Clifford invariant of $(W,F)$: 
 In this case $q \equiv 1 \pmod{4}$ so 
 $-\zeta _{q+1}^2 $ is a primitive $q+1$st root of unity 
 and hence all characters of 
 degree $q-1$ of the group $\psl_2(q)$ extend to characters of 
 $\text{PGL} _2(q) $ with the same character field (see \cite[Table III]{steinberg} 
 for a character table) and of Schur index 1 (see \cite{gow}).
 So  $(W,F)$ is also an orthogonal representation of 
 $\text{PGL} _2(q) $ and
 restricting $(W,F)$ to $B$, we obtain the orthogonal sum of two 
 isometric spaces $(W,F) \cong (V_1,F_1) \perp (V_2,F_2)$ because the 
 normalizer of $B$ in  $\text{PGL} _2(q) $ interchanges $V_1$ and $V_2$.
 By Example \ref{borel} we have 
 $\dd (V_i,F_i) = p $ if $n\equiv 2 \pmod{4} $ and $p \equiv 1 \pmod{4}$ 
 and $\dd (V_i,F_i) = 1 $ otherwise ($i=1,2$). In both cases 
 $(\dd (V_1,F_1) , \dd (V_2,F_2) ) = 1 \in \br (\MQ ) $ and so 
by Example \ref{ex:orthsum} 
$\c (W,F) = \c(V_1,F_1) \c(V_2,F_2) = \c(V_1,F_1)^2 = 1 $.

\subsection{The characters $\chi _i$, $i$ even}  \label{secchi}

 For even $i$, the character $\chi _i$ is a non-faithful character of 
 even degree $q+1$ with totally real character field $K$ and Schur index 1. 
 As before we restrict $\chi _i$ 
to the Borel subgroup  and obtain 
$$\chi _i |_{B} = \psi _1 + \psi _2 + \alpha + \overline{\alpha } $$ 
where $\psi _1,\psi _2 $ are as in \ref{sectheta} and $\alpha $ is a complex linear
character of $B$. 
Comparing character values we obtain that $\alpha  (y) = \zeta _{q-1}^i $ 
for a suitably chosen generator $y$ of $ C_{(q-1)/2}  \leq B$. 
In particular $\MQ(\alpha ) = \MQ (\zeta _{q-1}^i ) = K(\sqrt{\vartheta _{q-1}^{(2i)}-2})$
and hence Remark \ref{three} (c) tells us that 
$\dd (\alpha ) = \vartheta _{q-1}^{(2i)}-2 $.
The discriminant of $\psi _1$ and $\psi _2$ behave as in \ref{sectheta} 
and hence we compute the discriminant 
$\dd (\chi _i) = \varepsilon (\vartheta _{q-1}^{(2i)}-2 ) q $.

 \subsection{The characters $\xi_1$, $\xi_2$ for $q\equiv 1\pmod{4}$.}
 
 Assume that $q=p^n\equiv 1 \pmod{4}$. 
 Then the two characters $\xi _1$ and $\xi _2$ of odd degree $\frac{q+1}{2}$ 
 factor through $\PSL_2(q)$ and have a 
 totally real character field $K= \MQ(\chi _1) = \MQ(\chi _2) = 
 \MQ (\sqrt{q}) $.


 \begin{Proposition}\label{prop}
	 There are the following two possibilities for $\c(\xi _1)= \c (\xi_2)$:
 \begin{center}
 \begin{tabular}{|c||c|c||c|c|c|c|}
 \hline 
 & \multicolumn{2}{|c||}{$n$ even }
 & \multicolumn{4}{|c|}{$n$ odd } \\
 \hline
 $q \pmod{16} $ & $1$  & $9$ & $1$ & $-3$ & $9$ & $5 $ \\ 
 \hline 
 $ \star $ \  $\c(\xi _1) = \c(\xi_2) $  & $1$ & $[\mathcal{Q}_{p,\infty } ]$ & 
 $1$ &  $1$ & $[\mathcal{Q}_{\infty _1,\infty_2 } ] $ & $[\mathcal{Q}_{\infty _1,\infty_2 } ] $ \\
 \phantom{$\star $} \ $\c(\xi _1) = \c(\xi_2) $  & $[{\mathcal Q}_{2,p}]$ & $[\mathcal{Q}_{2,\infty } ]$ & 
 $[{\mathcal Q}_{\wp_1,\wp_2}]$ &  $[{\mathcal Q}_{2,\sqrt{p}}]$ & $[\mathcal{Q}_{\infty _1,\infty_2,\wp _1,\wp_2 } ] $ & $[\mathcal{Q}_{\infty _1,\infty_2,2,\sqrt{p} } ] $ \\
 \hline
 \end{tabular} 
\end{center}
 Here, for $p\equiv 1\pmod{8} $ and $n$ odd,
 $\wp_1 $ and $\wp_2$ denote the two places of $K=\MQ(\sqrt{p})$ 
 that divide $2$. 
 \end{Proposition}

 \begin{proof}
 Let $\xi $ be one of $\xi _1$ or $\xi _2$, 
 $K=\MQ(\sqrt{q})$ and $W$ the $KG$-module affording the character $\xi $.
 Since $\F_G(W)$ always contains a totally positive definite form, we know that 
 $\c(\xi ) \otimes \MR  = 1 $  if $q \equiv 1,-3 \pmod{16}$ and
 $\c(\xi ) \otimes \MR  \neq 1 $ otherwise, for all real places of $K$. 
 If $K\neq \MQ $ then 
 $\xi _1 $ and $\xi _2$ are Galois conjugate and so are 
 $\c (\xi_1) $ and $\c(\xi _2)$. 
 The outer automorphism of $G$ interchanges $\xi _1$ and $\xi _2$ which also 
 shows that $\c(\xi _1) = \c( \xi _2)$, so this algebra is 
 stable under the Galois group of $K$. 
 Moreover the only possible finite primes of $K$  that ramify in 
 $c (\xi )$ are those dividing $p$ or $2$. 
 This is seen as follows: 
 The representation $\xi $ is irreducible modulo all other primes 
 $\ell $ (see \cite[Section 9.3]{bonnafe})
so in particular there is a $G$-invariant lattice $L$ in  $W$  whose determinant 
is not divisible by $\ell $ and hence $\ell $ does not ramify in $\c(W,F)$ 
by Remark \ref{detCO}.
Noting that $2$ is decomposed in $K$ if and only if $n$ is odd and 
$p\equiv 1 \pmod{8}$,
 we are left with the possibilities for $\c (\xi )$ as stated.
\end{proof}

\begin{Lemma} \label{lemmastar}
	$\c(\xi _1) = \c(\xi _2)$ is given in 
	line $\star $  of Proposition \ref{prop}.
\end{Lemma}

\begin{proof}
	Let $\xi $ be one of $\xi_1$ or $\xi _2$.
	By Proposition \ref{prop} it suffices to show that the 
	primes of $K$ that divide $2$ do not ramify in $\c(\xi )$.
	So let $\wp $ be a prime ideal of $K$ that contains $2$ and let 
	$R_{\wp }$ be the valuation ring in the completion $K_{\wp }$ 
	(so $R_{\wp }\cong \MZ_2$ if $q\equiv 1 \pmod{8} $ and $R_{\wp } \cong \MZ_2 [\zeta _3]$ if 
	$q\equiv 5 \pmod{8} $). 
	By \cite[Theorem VII.12 and Theorem VII.4]{plesken83} the image of $R_{\wp }G$ 
	in $\End (K_{\wp } \otimes W)$ is isomorphic to 
	$$ \Delta _{\xi } (R_{\wp }G) = \left( \begin{array}{cc} 
			R_{\wp } & 2R_{\wp }^{1\times (q-1)/2} \\ 
	R_{\wp }^{(q-1)/2 \times 1} & R_{\wp }^{(q-1)/2 \times (q-1)/2} \end{array} \right) .$$
			In particular the $R_{\wp }G$-lattices in $K_{\wp } \otimes W$ form a chain 
$$\ldots \supset L' \supset L \supset 2L' \supset 2L \ldots $$
with $L'/L \cong R_{\wp }/2R_{\wp }$. 
If $F\in \F_G(W)$ is non-degenerate and $L$ is $G$-invariant, then also
its dual lattice is $G$-invariant. This shows that there is some $F\in \F_G(W)$ such that 
$L'$ is the dual lattice of $L$. But then $Q_F(L) \subseteq R_{\wp }$ because 
otherwise the even sublattice of $L$ would be a $G$-invariant sublattice of index $2$
in $L$. So $L$ is a  semi-regular quadratic $R_{\wp }$-module in $(K_{\wp }\otimes W,F)$ 
and by Corollary \ref{semireg} this implies that $\c (K_{\wp }\otimes W,F) = 1$.
\end{proof}

Note that for $n=1$ and $n=2$ it is also possible to deduce this lemma using the
character theoretic method from \cite{nebeexpmath} (see \cite[Section 6.4]{braun}).

 \section{The orthogonal character table of $\SL_2(2^n)$}\label{section:sl22}

 We now assume that $q=2^n$ with $n\geq 2$  and put $G:= \SL_2(q)$. 
 Then the ordinary character table of $G$ is given in 
 \cite[Theorem 38.2]{dornhoff}:

\begin{Theorem}[\protect{\cite[Theorem 38.2]{dornhoff}}]
	  Let $\nu$ be a generator of $\MF_q^\times$ and consider the elements
	    \begin{equation*}
	       1=\begin{pmatrix}1&0 \\ 0&1 \end{pmatrix}, ~ c:=\begin{pmatrix}1&0 \\ 1&1 \end{pmatrix} , ~ a := \begin{pmatrix} \nu & 0 \\ 0& \nu^{-1} \end{pmatrix}
	      \end{equation*}
        of $G$.
  The group also contains an element $b$ of order $q+1$.
    The character table of $G$ is
			      {
		        \begin{center}
  \def\arraystretch{1.3}
      \begin{tabular}{c|cccc}
        & $1_G$ & $c$ & $a^\ell$ & $b^m$\\
        \hline $\1$ & $1$ & $1$ & $1$ & $1$ \\
        $\psi $ & $q$ & $0$ & $1$ & $-1$\\
        $\chi_i$ & $q+1$ & $1$ & $\zeta_{q-1}^{i\ell} + \zeta_{q-1}^{-i \ell}$ & $0$ \\
        $\theta_j$ & $q-1$ & $-1$ & $0$ & $-\zeta_{q+1}^{jm} - \zeta_{q+1}^{-jm}$ \\
        \end{tabular}
  \end{center}
     }
  where $1\leq i \leq \frac{q-2}{2}$, $1\leq j \leq \frac{q}{2}$, $1\leq \ell \leq \frac{q-2}{2}$, $1\leq m \leq \frac{q}{2}$.
    \end{Theorem}

In contrast to the odd characteristic case all characters have 
totally real character field and Schur index 1. 

     \begin{Theorem}[Orthogonal representations of $\SL_2(2^n)$]\label{sl22main}
      Let $q=2^n$, $n\geq 2$ and $G=\SL_2(q)$. 
      Then the non-trivial irreducible characters of $G$ have $G$-invariant bilinear forms with the following algebraic invariants.

         \begin{center}
	     \def\arraystretch{1.3}
        \begin{tabular}{|l|l|}
       \hline Character & Invariant \\
     \hline \hline $\psi $ & $\dd(\psi)= q+1 $\\
				   \hline $\chi_i$, $1\leq i \leq \frac{q-2}{2}$ & $\c (\chi _i ) = \begin{cases} 1  \in \br(\MQ(\chi_i))  & \text{ if } n \text{ is odd, see Theorem \ref{even} } \\
			   \text{ see  Theorem \ref{chii}} &  \text{ if } n  \text{ is even } \end{cases} $ \\
\hline $\theta_j$, $1\leq j \leq \frac{q}{2}$ & $\c(\theta_j)=\begin{cases}
                                                              (-1,-1) \in \br(\MQ(\sqrt{5})) & \text{ if } q=4, \\
                                                              1 \in \br(\MQ(\theta_j)) & \text{ if } q\geq 8.
                                                             \end{cases}$\\
          \hline
      \end{tabular}
        \end{center}
\end{Theorem}

\begin{proof}
	For the Steinberg character $\psi $ we again have that 
	$\psi + \1 $ is the character of a 2-transitive permutation representation.
	In particular $\dd (\psi ) = \dd (\MA_{q}) = q+1 $. 
	For the characters $\theta _j$ of degree $q-1$ we note that the restriction 
	of these characters to the normalizer $B\cong C_2^n \rtimes C_{q-1} $ of the 
	Sylow-2-subgroup of $G$ is the character of an irreducible rational monomial representation
	$V$. So $V$ has an orthonormal basis and hence 
	$\c(\theta  _j) = \c ( \MI _{q} \otimes K)  $ is given in Example \ref{ex:An}.
\end{proof}

To describe the Clifford invariant of the characters $\chi _i$  of degree $q+1$ 
note that for the infinite places of $K$ 
the invariant $[c (\chi _i)  \otimes _K \MR ] \in \br (\MR ) $   is non-trivial 
if and only if $q=4$, because in all other cases, the character degree is 
$1 \pmod{8} $. 

For the odd finite primes of $K$, the Clifford invariant of $\chi _i$ is 
given in the next theorem:

\begin{Theorem} \label{chii}
	Let $1\leq i \leq (q-2)/2$, $K = \MQ (\chi _i) = \MQ [\vartheta _{q-1}^{(i)} ] $, and let 
	$\wp $ be some maximal ideal of $\MZ_K$ such that $\wp \cap \MZ = p \MZ $ for some
	odd prime $p$. Then $[c (\chi _i) \otimes K_{\wp }]  \in \br(K_{\wp } )$ 
	is not trivial if and only if 
	$$ \text{ (i) } p\equiv \pm 3 \pmod{8} ,  \text{ and (ii) }
	(q-1)/(\gcd(q-1,i)) \text{ is a power of } p. $$
\end{Theorem}

\begin{proof}
	We first note that condition (ii) implies  that $p$ divides $q-1$.
	If condition (ii) is not fulfilled, then the reduction of 
	$\chi _i$ modulo $\wp $ is an irreducible Brauer character (see for instance 
	\cite{Burkhardt}). In particular the orthogonal $K_{\wp }G$-module $V$
	affording the character $\chi _i$ contains an (even) unimodular $R_{\wp }$-lattice. 
	So Corollary \ref{odddetur} tells us that 
	 $[c (\chi _i) \otimes K_{\wp }] = 1 \in \br(K_{\wp } )$. 
	 If the condition (ii) is satisfied, then $\wp $ is the unique prime ideal of $K$ that
	 contains $p$, the extension $K_{\wp }/\MQ_p$ is totally ramified,
	 and (again by \cite{Burkhardt}) the $\wp $-modular Brauer tree 
	 of the block containing $\chi _i$ is given as 
	 \begin{center}
	 \setlength{\unitlength}{0.8mm}
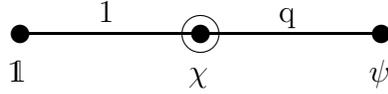
	 \begin{picture}(100, 15)
 \put(20,10){\circle*{3}}
 \put(50,10){\circle*{3}}
 \put(50,10){\circle{6}}
 \put(80,10){\circle*{3}}
 \drawline(20,10)(80,10) 
 \put(18,2){$\1$} 
 \put(78,2){$\psi $} 
 \put(48,2){$\chi $} 
 \put(33,12){1} 
 \put(63,12){q} 
 \end{picture}
 \end{center}
where the multiplicity of the exceptional vertex $\chi $ is 
$\frac{p^a-1}{2}$  with $a=\nu _p(q-1)$.
In particular \cite[Theorem (VIII.3)]{plesken83} yields that the 
$R_{\wp }$-order $R_{\wp }G $ acts on $V$ as 
$$ \Delta _{\chi_i } (R_{\wp }G) = \left( \begin{array}{cc} 
		R_{\wp } & \wp R_{\wp}^{1\times q} \\ 
R_{\wp }^{q \times 1} & R_{\wp }^{q \times q} \end{array} \right). $$
		As in the proof of Lemma \ref{lemmastar} the $R_{\wp }G $-invariant
		lattices in $V$ form a chain: 
		$$\ldots \supset L' \supset L \supset \wp L' \supset \wp L \ldots $$
		with $L'/L \cong R_{\wp }/\wp R_{\wp } $. So there is a $G$-invariant
		form $F$ on $V$ such that $L' = L^{\# }$, in particular the
		$\wp $-adic valuation of the determinant of $L$ is 1. 
		Choose $(b_1,\ldots , b_{q}) \in L^q$ such that 
		the images form a basis $\overline{B}$ of $L/\wp L'$ and put 
		$W := \langle b_1,\ldots , b_q \rangle _{K_{\wp }} \leq V$. 
	The modular representation  $L/\wp L' $ is isomorphic to the 
	$\wp $-modular reduction of the Steinberg module $\psi $. 
	In particular the determinant of the Gram matrix of $\overline{B}$ is 
	$\overline{q+1}\in \MZ/p\MZ \cong R_{\wp} /\wp R_{\wp}$.
	As $\wp $ is odd and $q+1 \in R_{\wp}^{\times } $ this 
	gives the discriminant of the bilinear $K_{\wp } $-module 
	$$\dd (W,F_{|W}) = (q+1) (K_{\wp }^{\times })^2  = 2 (K_{\wp }^{\times })^2$$
	because $q+1 \equiv 2 \pmod{p} $ since $p$ divides $q-1$. 
	We can now apply Corollary \ref{primedet} to conclude that 
	the Clifford invariant of $(V,F)$ is non-trivial, if and only if $2$ is
	not a square in $K_{\wp }$, if and only if $2$ is not a square in $\MF _p = R_{\wp}/\wp $
	which is equivalent to condition (i) by quadratic reciprocity.
\end{proof}

\begin{Theorem} \label{even}
	If $q=2^n$ and $n$ is odd then $\c (\chi _i) = 1 \in \br (\MQ(\chi _i)) $ for all
	$1\leq i \leq \frac{q-2}{2}$. 
\end{Theorem}

\begin{proof}
	Let $M:= \MQ_2[\zeta _{2^n-1}]$ be the unramified extension of $\MQ_2$ of degree $n$.
	Then $M$ is a splitting field for $G$. 
	Moreover the $M$-representation $V_M$ affording the character $\chi _i$ is induced up 
	from a linear $M$-representation of the normalizer $B=C_2^n\rtimes C_{2^n-1}$ of the
	Sylow-2-subgroup of $G$. 
	In particular $V_M$ is an irreducible monomial representation and hence the 
	standard form $F_M$ is $G$-invariant, so $(V_M,F_M) \cong \MI _{2^n+1} \otimes M $. 
	For $n\geq 3$ the dimension of $V_M$ is $\equiv 1 \pmod{8}$ and so by Example
	\ref{ex:An} the Clifford invariant of $(V_M,F_M)$ is trivial in $\br (M)$.
	Now let $K = \MQ(\chi _i) $, $(V,F)$ an orthogonal $KG$-module affording the
	character $\chi _i$, and let $\wp $ be some prime ideal of $K$ dividing $2$.
	As $K\subseteq  \MQ [\zeta _{2^n-1}]$ the completion of $K$ at $\wp $ is contained in $ M$ 
	and, by the same argument as before, 
	$(V\otimes M,F ) \cong (V_M , aF_M )$ for some non-zero $a\in M$.
	In particular $\c (V\otimes M,F ) = 1 $ in $\br(M)$.
	As $[M:\MQ _2] = n$ is assumed to be odd, also  $[M:K_{\wp }] $ is odd and hence 
	$\c(V\otimes K_{\wp },F)  = 1 $ in $\br(K_{\wp })$. 
	This argument shows that no even prime $\wp $ of $K$ ramifies in $c(V,F)$. 
	Also the real primes do not ramify because $\dim (V) \equiv 1 \pmod{8}$.
	So by Theorem \ref{chii} there is at most one prime ideal of $K$ that 
	ramifies in $c(V,F)$. But the number of ramified primes is even, which
	shows that $\c(\chi _i ) = 1$ in the Brauer group of $K$.
\end{proof}

Note that Theorem \ref{even} together with Theorem \ref{chii} implies the well known fact that 
if $n$ is odd then all primes $p$ dividing $2^n-1$ satisfy $p \equiv \pm 1 \pmod{8}$ 
(because then $2^{(n+1)/2}$ is a square root of $2$ modulo $p$). 

\begin{Remark} \label{open}
In the situation of Theorem \ref{chii} if 
$[c (\chi _i) \otimes K_{\wp }]  \in \br(K_{\wp } )$ is non-trivial and $q\neq 4$, 
then an odd number of
even primes of $K$ also ramify in $c(\chi _i)$. 
However, we did  not determine in general which even primes of $K$ ramify 
in $c(\chi _i)$ for the case that $n$ is even. 
Of course the same argument as in the proof of Theorem \ref{even} works if 
the primes above $2$ are decomposed in $\MQ (\zeta _{q-1}^{i} ) / \MQ (\vartheta _{q-1}^{(i)} )$.
\end{Remark}

 \bibliography{sources.bib}

\begin{thebibliography}{10}

\bibitem{bonnafe}
C.~Bonnaf\'{e}.
\newblock {\em {Representations of $\mathrm{SL}_2(\MF_q)$}}, volume~13 of {\em
  Algebra and applications}.
\newblock Springer-Verlag, Berlin-New York, 2011.

\bibitem{braun}
O.~Braun.
\newblock {\em {Orthogonal Representations of Finite Groups}}.
\newblock Dissertation. RWTH Aachen University, 2016.

\bibitem{Burkhardt}
R.~Burkhardt.
\newblock Die {Z}erlegungsmatrizen der {G}ruppen {${\rm PSL}(2,p^{f})$}.
\newblock {\em J. Algebra}, 40(1):75--96, 1976.

\bibitem{curtisreiner1}
C.~W. Curtis and I.~Reiner.
\newblock {\em {Methods of Representation Theory}}, volume~1.
\newblock John Wiley and Sons, Inc., New York, 1987.

\bibitem{dornhoff}
L.~Dornhoff.
\newblock {\em Group representation theory. {P}art {A}: {O}rdinary
  representation theory}.
\newblock Marcel Dekker, Inc., New York, 1971.
\newblock Pure and Applied Mathematics, 7.

\bibitem{gow}
R.~Gow.
\newblock Schur indices of some groups of {L}ie type.
\newblock {\em J. Algebra}, 42(1):102--120, 1976.

\bibitem{hasse}
H.~Hasse.
\newblock {\"Aquivalenz quadratischer Formen in einem beliebigen Zahlk\"orper}.
\newblock {\em J. reine u. angew. Mathematik}, 153:158--162, 1924.

\bibitem{janusz}
G.~J. Janusz.
\newblock Simple components of {$Q[{\rm SL}(2,\,q)]$}.
\newblock {\em Comm. Algebra}, 1:1--22, 1974.

\bibitem{Kneser}
M.~Kneser.
\newblock {\em Quadratische {F}ormen}.
\newblock Springer-Verlag, Berlin-New York, 2002.

\bibitem{lam73}
T.-Y. Lam.
\newblock {\em The algebraic theory of quadratic forms}.
\newblock WA Benjamin, 1973.

\bibitem{nebehabil}
G.~Nebe.
\newblock {\em Orthogonale {D}arstellungen endlicher {G}ruppen und
  {G}ruppenringe}.
\newblock Aachener Beitr\"age zur Mathematik 26. Wissenschaftsverlag Mainz,
  Aachen, 1999.

\bibitem{nebeexpmath}
G.~Nebe.
\newblock Invariants of orthogonal {G}-modules from the character table.
\newblock {\em Experimental Mathematics}, 9(4):623--629, 2000.

\bibitem{plesken83}
W.~Plesken.
\newblock {\em Group rings of finite groups over p-adic integers}.
\newblock Number 1026 in Lecture Notes in Mathematics. Springer-Verlag, Berlin,
  1983.

\bibitem{reiner}
I.~Reiner.
\newblock {\em Maximal {O}rders}, volume~38.
\newblock Academic Press London, 1975.

\bibitem{scharlauquadratic}
W.~Scharlau.
\newblock {\em Quadratic and {H}ermitian forms}.
\newblock Springer-Verlag, Berlin-New York, 1985.

\bibitem{schurpsl}
I.~Schur.
\newblock {Untersuchungen \"uber die Darstellungen der endlichen Gruppen durch
  gebrochene lineare Substitutionen}.
\newblock {\em J. Reine Angew. Math.}, 132:85--137, 1907.

\bibitem{steinberg}
R.~Steinberg.
\newblock The representations of {${\rm GL}(3,q), {\rm GL} (4,q), {\rm PGL}
  (3,q)$}, and {${\rm PGL} (4,q)$}.
\newblock {\em Canadian J. Math.}, 3:225--235, 1951.

\bibitem{turull}
A.~Turull.
\newblock Schur index two and bilinear forms.
\newblock {\em Journal of Algebra}, 157(2):562--572, 1993.

\end{thebibliography}
\end{document}